\documentclass[a4paper,11pt]{article}

\usepackage[margin=1in]{geometry}
\usepackage{setspace}
\usepackage{amsmath}
\usepackage{amssymb}
\usepackage{amsthm}
\usepackage{xcolor}
\usepackage{hyperref}

\newtheorem{theorem}{Theorem}[section]
\newtheorem{proposition}[theorem]{Proposition}
\newtheorem{corollary}[theorem]{Corollary}
\newtheorem{lemma}[theorem]{Lemma}
\newtheorem{definition}[theorem]{Definition}

\hypersetup{
    hidelinks,
    pdftitle={Non-K\"ahler Critical Hermitian Metrics of the Dinew--Popovici Functional},
    pdfauthor={Hanwen Liu}
}

\begin{document}

\title{\texorpdfstring{\textbf{Non-K\"ahler Critical Hermitian Metrics of the Dinew--Popovici Functional}}{Non-K\"ahler Critical Hermitian Metrics of the Dinew--Popovici Functional}}
\author{Hanwen Liu\\
{\small Mathematics Institute, University of Warwick}\\
{\small \href{mailto:hanwen.liu@warwick.ac.uk}{\texttt{hanwen.liu@warwick.ac.uk}}}}
\date{}

\maketitle

\begin{abstract}
The Dinew--Popovici functional is an energy functional for Hermitian symplectic metrics in a fixed Aeppli cohomology class. Its vanishing characterizes the K\"ahler metrics in that class, providing a variational approach to K\"ahler geometry. Dinew and Popovici proved that every critical point is K\"ahler in complex dimension three. We give a negative answer to Erfan Soheil's question about higher dimensions by constructing non-K\"ahler critical metrics on products of two K\"ahler surfaces $(S_1,\eta_1)$ and $(S_2,\eta_2)$. These metrics are critical under every variation on the product 4-fold. We characterize the critical product metrics and prove that non-K\"ahler critical products exist in the Aeppli class $[\eta_1+\eta_2]_A$ precisely when the canonical bundles of the two surface factors are smoothly trivial. As a by-product, we develop a geometric flow driven by the torsion tensor and converging to the fixed K\"ahler background when this background metric has nonnegative holomorphic bisectional curvature.
\end{abstract}

\begin{center}

\end{center}

\tableofcontents
\onehalfspacing
\raggedbottom

\section{Introduction and Background}\label{section_introduction}

Dinew and Popovici introduced an energy functional on the Hermitian symplectic metrics in a fixed Aeppli cohomology class \cite{DinewPopovici}. For a Hermitian symplectic form $\omega$ on a compact complex manifold $X$, let $\rho_\omega$ be the $(2,0)$-form of least $L^2(\omega)$ norm such that $\omega+\rho_\omega+\bar\rho_\omega$ is closed. Their functional is defined by $\mathcal{F}_X(\omega)=\|\rho_\omega\|_{L^2(\omega)}^2$. The equality $\mathcal{F}_X(\omega)=0$ holds precisely when $\omega$ is K\"ahler. Thus, the functional provides a variational approach to finding K\"ahler metrics in a prescribed Aeppli class.

In complex dimension three, Dinew and Popovici proved that every critical point of $\mathcal{F}_X$ is K\"ahler \cite{DinewPopovici}. Consequently, in that dimension, criticality already forces the energy to vanish. Erfan Soheil subsequently studied the functional in higher dimensions and under holomorphic deformations, and asked whether its critical points are still exactly the K\"ahler metrics in the given Aeppli class when the complex dimension exceeds three \cite[Question~1.5]{Soheil}.

In this article, we give a negative answer to Soheil's question in complex dimension four. Our counterexamples are non-K\"ahler Hermitian metrics on products of two compact K\"ahler surfaces $(S_1,\eta_1)$ and $(S_2,\eta_2)$. On the product K\"ahler background $(X,\eta)=(S_1\times S_2,\eta_1+\eta_2)$, we construct positive forms $\omega=\omega_1+\omega_2$, with $\omega_i\in[\eta_i]_A$, which are critical for $\mathcal{F}_X$ under every Aeppli variation. In particular, criticality includes variations which do not preserve the product decomposition. These examples have positive energy, although their Aeppli class contains the zero-energy K\"ahler representative $\eta$.

The counterexamples arise from a characterization of the full critical points which are product metrics. Let $\rho_i$ be the canonical completion of $\omega_i$. Theorem~\ref{critical_products} shows that a non-K\"ahler product $\omega=\omega_1+\omega_2$ is critical in the full Aeppli class $[\eta]_A$ if and only if it holds simultaneously that
$$
\begin{cases}
2\rho_1\wedge\bar\rho_1=\omega_1^2,\\
2\rho_2\wedge\bar\rho_2=\omega_2^2.
\end{cases}
$$
Each equality is pointwise on the corresponding surface. Theorem~\ref{existence} then proves that such a critical product exists in the prescribed class precisely when the canonical bundles $K_{S_1}$ and $K_{S_2}$ are smoothly trivial, equivalently, when $c_1(K_{S_i})=0$ in $H^2(S_i,\mathbb Z)$ for each $i\in\{1,2\}$. In particular, the construction applies to products of K3 surfaces and complex 2-tori for every choice of the background K\"ahler forms.

The existence proof begins with a nowhere-zero $\partial$-exact $(2,0)$-form on each surface. Buchdahl's positivity theorem \cite{Buchdahl} and the Hermitian Monge--Amp\`ere theorem \cite{TosattiWeinkove} then produce a positive Hermitian form with the required volume density. The pointwise equalities above turn the full first variation of $\mathcal{F}_X$ into the integral of an exact form. This establishes criticality under all admissible variations and completes the construction of the counterexamples.

We also study the geometry of these critical metrics. Their Hessians have infinite positive and negative indices, even under variations preserving the exterior torsion form $d\omega$ (Corollary~\ref{product_infinite_index}). The torsion flow~\eqref{torsion_flow} selects a negative Hessian direction whose value is expressed explicitly through the dissipation of the two surface completion energies (Theorem~\ref{flow_instability}). The associated evolution theory is developed in Appendix~\ref{appendix_flow}: on a compact K\"ahler background $(X,\eta)$ with nonnegative holomorphic bisectional curvature, the flow exists for all time and converges smoothly to $\eta$ for every smooth positive Hermitian initial form and every positive damping constant.

The paper is structured as follows: Section~\ref{section_introduction} is this introduction. Section~\ref{section_products} defines the Dinew--Popovici functional and classifies the full critical points which are product metrics; Section~\ref{section_existence} proves the existence criterion for the non-K\"ahler critical products; Section~\ref{section_flow} proves the flow Hessian formula and the corresponding instability results; and Section~\ref{section_potentials} studies the first and second variations of $\mathcal{F}_X$ along potential directions on K\"ahler 4-folds $(X,\eta)$. The final main section summarizes the critical-point and instability results. Appendix~\ref{appendix_flow} proves the long-term existence and uniqueness results for the torsion flow, including the convergence theorem under nonnegative holomorphic bisectional curvature.

\section{Critical Metrics on Products of Surfaces}\label{section_products}

Throughout this article, all manifolds are smooth, compact, connected and without boundary. A symbol such as $X$ denotes a complex manifold. When a K\"ahler form is prescribed, we write $(X,\eta)$ for the complex manifold together with that form. We identify a Hermitian metric with its positive real fundamental $(1,1)$-form, and use $\omega$ for a Hermitian form which is allowed to vary. Thus, the form $\eta$ in a K\"ahler background $(X,\eta)$ is fixed even when the Hermitian form $\omega$ on $X$ varies.

All forms and functions are smooth unless another regularity is specified. We use the convention $d^c=i(\bar\partial-\partial)$, so that $dd^c=2i\partial\bar\partial$. Exterior powers always mean wedge powers. For a Hermitian $(1,1)$-form $\omega$ on a complex manifold $X$ of dimension $n$, we write $\mu_\omega=\omega^n/n!$ for the Riemannian volume measure and $L^2(\omega):=L^2(X;\mu_\omega)$. For differential forms, this notation includes the pointwise inner product induced by $\omega$ on the exterior powers of the cotangent bundle. In particular, for differential forms $\alpha,\beta$ of the same degree, we have that
$$
\langle\alpha,\beta\rangle_{L^2(\omega)}
=\int_X\langle\alpha,\beta\rangle_\omega\,\mu_\omega.
$$
The Hermitian inner product is taken to be complex-linear in the first argument. We suppress pullbacks from the factors of a product whenever the source factor of each form has been specified.

\subsection{The functional and its variations}

Let $(X,\eta)$ be a K\"ahler manifold of complex dimension $n\geq2$. We write $\Omega^{p,q}(X)$ for the smooth differential forms of bidegree $(p,q)$ on the complex manifold $X$, and $\Omega^{1,1}(X,\mathbb R)$ for the real $(1,1)$-forms. The real Aeppli cohomology of $X$ in bidegree $(1,1)$ is
$$
H_A^{1,1}(X,\mathbb R)
=\frac{\{\omega\in\Omega^{1,1}(X,\mathbb R):\partial\bar\partial\omega=0\}}
{\{\partial\bar v+\bar\partial v:v\in\Omega^{1,0}(X)\}}.
$$
A real $(1,1)$-form $\omega$ is said to be pluriclosed if $dd^c\omega=0$. For a pluriclosed form $\omega$, we denote its Aeppli cohomology class by $[\omega]_A$. For the fixed K\"ahler form $\eta$, the positive representatives of its Aeppli class are precisely the forms
$$
\omega=\eta+\partial\bar v+\bar\partial v>0,
$$
where $v$ is a smooth $(1,0)$-form on $X$. The inequality means that $\omega$ is positive definite as a Hermitian form at every point of $X$.

\begin{definition}\label{completion_definition}
Let $\omega$ be a Hermitian form on the complex manifold $X$. We say that $\omega$ is Hermitian symplectic if there exists a $(2,0)$-form $\rho$ on $X$ such that $\partial\rho=0$ and $\bar\partial\rho=-\partial\omega$. We call such a $2$-form $\rho$ a completion of $\omega$. Equivalently, $\rho$ is a completion of $\omega$ when the real 2-form $\omega+\rho+\bar\rho$ is closed. Among all completions of the fixed form $\omega$, the completion of least $L^2(\omega)$ norm is called the canonical completion of $\omega$ and is denoted by $\rho_\omega$.
\end{definition}

Every Hermitian symplectic form $\omega$ is pluriclosed, because applying $\bar\partial$ to the completion equation $\bar\partial\rho=-\partial\omega$ gives $\partial\bar\partial\omega=0$. Every positive form $\omega=\eta+\partial\bar v+\bar\partial v$ in $[\eta]_A$ is Hermitian symplectic: the form $\partial v$ completes $\omega$, and the corresponding closed real 2-form is $\eta+d(v+\bar v)$. More generally, if $\omega$ is Hermitian symplectic with completion $\rho$, then every positive form $\omega+\partial\bar u+\bar\partial u$ in $[\omega]_A$ is Hermitian symplectic with completion $\rho+\partial u$.

For a fixed Hermitian symplectic form $\omega$, the difference of any two completions of $\omega$ is a closed holomorphic 2-form on $X$. Conversely, adding a closed holomorphic 2-form to a completion of $\omega$ gives another completion of $\omega$. Since $(X,\eta)$ is K\"ahler, every holomorphic form on $X$ is closed. Thus, the set of completions of $\omega$ is an affine space whose translation space is the finite-dimensional vector space of holomorphic 2-forms on $X$. Minimizing the positive quadratic form $\|\rho\|_{L^2(\omega)}^2$ on this affine space gives a unique canonical completion. The defining orthogonality condition is
$$
\langle\rho_\omega,h\rangle_{L^2(\omega)}=0
$$
for every holomorphic 2-form $h$ on $X$.

For a Hermitian symplectic metric $\omega$ on the complex manifold $X$, the Dinew--Popovici energy of $\omega$ is
\begin{equation}\label{functional}
\mathcal{F}_X(\omega)=\|\rho_\omega\|_{L^2(\omega)}^2
=\frac{1}{(n-2)!}\int_X\rho_\omega\wedge\bar\rho_\omega\wedge\omega^{n-2}.
\end{equation}
For each fixed Aeppli class containing a Hermitian symplectic metric, the Dinew--Popovici functional on that class assigns the value~\eqref{functional} to each positive representative of the class. Unless the class is specified otherwise, the class considered below is $[\eta]_A$. The notation $\mathcal{F}_X$ retains the complex manifold $X$, but does not include the background form $\eta$, because the value $\mathcal{F}_X(\omega)$ depends on $\omega$ and the complex structure of $X$, not on the auxiliary choice of $\eta$.

An admissible variation at a positive representative $\omega$ of a fixed Aeppli class is a real $(1,1)$-form
$\psi=\partial\bar u+\bar\partial u$,
where $u$ is an arbitrary smooth $(1,0)$-form on $X$. A real $(1,1)$-form expressible in this way is called Aeppli-exact. For sufficiently small real $t$, the form $\omega+t\psi$ is positive and lies in the same Aeppli class as $\omega$. We write
$$
(D\mathcal{F}_X)_\omega\psi
=\left.\frac{d}{dt}\mathcal{F}_X(\omega+t\psi)\right|_{t=0}.
$$
A form $\omega$ is a critical point of $\mathcal{F}_X$ in the full Aeppli class $[\omega]_A$ if $(D\mathcal{F}_X)_\omega\psi=0$ for every admissible variation $\psi$. The word full specifies that no product, symmetry, or potential restriction is imposed on the $(1,0)$-form $u$ defining $\psi$.

We denote the real affine Hessian of $\mathcal{F}_X$ at $\omega$ by $H_\omega=D^2\mathcal{F}_X|_\omega$. For two admissible variations $\psi_1,\psi_2$, this means
$$
H_\omega(\psi_1,\psi_2)
=\left.\frac{\partial^2}{\partial t_1\partial t_2}
\mathcal{F}_X(\omega+t_1\psi_1+t_2\psi_2)\right|_{t_1=t_2=0}.
$$
The adjective affine refers to the expression of the fixed Aeppli class as a translate of the vector space of Aeppli-exact real $(1,1)$-forms. If $\omega$ is a critical point and $\omega(t)$ is any smooth curve of positive representatives in $[\omega]_A$ with $\omega(0)=\omega$, then the usual second derivative along the curve is precisely
$$
\left.\frac{d^2}{dt^2}\mathcal{F}_X(\omega(t))\right|_{t=0}
=H_\omega(\dot\omega(0),\dot\omega(0)).
$$
Indeed, the additional term in the chain rule is $(D\mathcal{F}_X)_\omega\ddot\omega(0)$, which vanishes by the criticality of $\omega$ in its full Aeppli class.

We shall use the smooth dependence of $\rho_\omega$ on the positive Hermitian form $\omega$. Let $\bar\partial_\eta^*$ be the adjoint of $\bar\partial$ in $L^2(\eta)$, and let $G_\eta$ be the inverse of the Dolbeault Laplacian $\bar\partial\bar\partial_\eta^*+\bar\partial_\eta^*\bar\partial$ on the orthogonal complement of its kernel, extended by zero on its kernel. Since $-\partial\omega$ is $\bar\partial$-exact, Hodge decomposition and the K\"ahler identities show that the form
$$
\sigma_\omega=-\bar\partial_\eta^*G_\eta\partial\omega
$$
satisfies $\partial\sigma_\omega=0$ and $\bar\partial\sigma_\omega=-\partial\omega$. The canonical completion $\rho_\omega$ is obtained by subtracting from $\sigma_\omega$ its $L^2(\omega)$ orthogonal projection onto the finite-dimensional space of holomorphic 2-forms. In a fixed basis of that space, the coefficients of this projection solve a positive definite linear system whose matrix and right-hand side depend smoothly on $\omega$. This proves the required smooth dependence.

For the first variation of $\mathcal{F}_X$, any smooth family of completions agreeing with the canonical completion at the initial time may be used. Let $\omega(t)$ be an admissible smooth curve with $\omega(0)=\omega$, and let $\rho(t)$ be any smooth family of completions of $\omega(t)$ satisfying $\rho(0)=\rho_\omega$. At every $t$, the difference $\rho(t)-\rho_{\omega(t)}$ is a holomorphic 2-form. Denote the derivative of this difference at $t=0$ by $h$. The form $h$ is holomorphic, and $\rho(0)-\rho_{\omega(0)}=0$. Thus, the difference between the first derivatives of the two squared norms is $2\operatorname{Re}\langle h,\rho_\omega\rangle_{L^2(\omega)}$. The orthogonality defining $\rho_\omega$ makes this real number zero. Therefore, the chosen family $\rho(t)$ and the canonical family $\rho_{\omega(t)}$ give the same first derivative of the energy at time $t=0$.

In complex dimension four, fix a Hermitian symplectic form $\omega$, write $\rho=\rho_\omega$, and let $\psi=\partial\bar u+\bar\partial u$ be an admissible variation. The family $\rho+t\partial u$ completes $\omega+t\psi$ and equals the canonical completion at $t=0$. Differentiating~\eqref{functional} with this family gives
\begin{equation}\label{first_variation}
(D\mathcal{F}_X)_\omega\psi
=\int_X\left(\operatorname{Re}\left(\partial u\wedge\bar\rho\wedge\omega^2\right)
+\rho\wedge\bar\rho\wedge\psi\wedge\omega\right).
\end{equation}
The first term in~\eqref{first_variation} differentiates the two completion factors, and the second term differentiates the factor $\omega^2/2$ in the energy density.

\subsection{Surface completions and product energy}

Let $(S_i,\eta_i)$ be K\"ahler surfaces for $i\in\{1,2\}$. For each index $i$, choose a smooth $(1,0)$-form $v_i$ on $S_i$ such that
\begin{equation}\label{surface_metrics}
\omega_i=\eta_i+\partial\bar v_i+\bar\partial v_i>0.
\end{equation}
We use the notation
\begin{equation}\label{surface_data}
\rho_i=\partial v_i,\qquad
\mu_i=\frac{1}{2}\omega_i^2,\qquad
\nu_i=\rho_i\wedge\bar\rho_i.
\end{equation}
Thus, $\mu_i$ is the positive volume form of the Hermitian metric $\omega_i$, and $\nu_i$ is a nonnegative top-degree form on the same complex surface $S_i$. We set $$\Theta_i=\omega_i+\rho_i+\bar\rho_i=\eta_i+d(v_i+\bar v_i).$$ The real 2-form $\Theta_i$ is closed and has the same de Rham cohomology class as $\eta_i$.

On the product complex manifold $X=S_1\times S_2$, we use the fixed K\"ahler form $\eta=\eta_1+\eta_2$ and the Hermitian form $\omega=\omega_1+\omega_2$. We also write $\rho=\rho_1+\rho_2$. The product metric $\omega$ lies in the Aeppli class $[\eta]_A$ because each surface metric $\omega_i$ lies in $[\eta_i]_A$.

\begin{lemma}\label{surface_completion}
For each K\"ahler surface $(S_i,\eta_i)$ and Hermitian form $\omega_i$ in~\eqref{surface_metrics}, the form $\rho_i=\partial v_i$ is the canonical completion of $\omega_i$. If $\hat\omega_i$ is another positive representative of $[\eta_i]_A$ satisfying $d\hat\omega_i=d\omega_i$, then $\rho_{\hat\omega_i}=\rho_i$. On the product background $(X,\eta)=(S_1\times S_2,\eta_1+\eta_2)$, the canonical completion of $\omega=\omega_1+\omega_2$ is $\rho=\rho_1+\rho_2$.
\end{lemma}

\begin{proof}
Fix $i\in\{1,2\}$. The form $\rho_i=\partial v_i$ completes $\omega_i$ because $\Theta_i$ is closed. For every holomorphic 2-form $h$ on $S_i$, we have that
$$
\langle\rho_i,h\rangle_{L^2(\omega_i)}
=\int_{S_i}\partial v_i\wedge\bar h=0.
$$
The last equality follows from Stokes' theorem and the closedness of $h$ on the K\"ahler background $(S_i,\eta_i)$. For arbitrary $(2,0)$-forms on a complex surface, the $L^2$ pairing equals the integral of the wedge product of the first form with the conjugate of the second form. Consequently, this pairing does not depend on the choice of Hermitian metric on the surface. The displayed orthogonality proves that $\rho_i$ is the canonical completion of $\omega_i$.

Now let $\hat\omega_i$ be the positive form specified in the lemma. The equality $d\hat\omega_i=d\omega_i$ makes the completion equations for $\hat\omega_i$ and $\omega_i$ identical. The difference $\rho_{\hat\omega_i}-\rho_i$ is therefore holomorphic. Both canonical completions are orthogonal to every holomorphic 2-form in the same metric-independent surface pairing. Applying this orthogonality to the difference itself shows that $\rho_{\hat\omega_i}-\rho_i=0$ vanishes identically.

A $(2,0)$-form on $S_1\times S_2$ is called pure if its covectors all come from one surface factor, and mixed if one covector comes from each surface factor. The two pure components and the mixed component are pointwise orthogonal with respect to the product metric $\omega_1+\omega_2$. Every holomorphic 2-form on the product is a sum of pullbacks of holomorphic 2-forms from the surfaces and products of holomorphic 1-forms from the two surfaces. The $L^2(\omega)$ pairing of $\rho_1+\rho_2$ with a pullback from $S_i$ contains the vanishing surface pairing just computed, multiplied by the volume of the other factor. The pairing with a mixed holomorphic form vanishes pointwise. Since $\rho_1+\rho_2$ completes $\omega_1+\omega_2$ and is orthogonal to every holomorphic 2-form on the product, it is the canonical completion of the product metric $\omega$.
\end{proof}

For each surface $(S_i,\eta_i)$, we denote the fixed K\"ahler volume and the completion energy by
\begin{equation}\label{surface_energies}
V_i=\frac{1}{2}\int_{S_i}\eta_i^2,\qquad
E_i=\int_{S_i}\nu_i.
\end{equation}
The number $V_i$ depends only on the prescribed K\"ahler class of $\eta_i$, whereas $E_i$ depends on the Hermitian form $\omega_i$. Expansion by type gives $\Theta_i^2/2=\mu_i+\nu_i$. Since the closed forms $\Theta_i$ and $\eta_i$ are cohomologous, integration gives
\begin{equation}\label{surface_volume}
\int_{S_i}\mu_i=V_i-E_i.
\end{equation}
The positivity of $\mu_i$ implies $E_i<V_i$, and the definition of $E_i$ implies $E_i\geq0$. Moreover, we have that $E_i=0$ exactly when $\rho_i=0$. The completion equation then gives $\partial\omega_i=0$, which is equivalent to $d\omega_i=0$ because $\omega_i$ is real. Conversely, if $\omega_i$ is K\"ahler, the zero form is a completion and is necessarily the canonical completion. Thus, we have $E_i=0$ if and only if the Hermitian form $\omega_i$ is K\"ahler on $S_i$.

By Lemma~\ref{surface_completion}, the functional $\mathcal{F}_X$ at the product metric $\omega_1+\omega_2$ is computed using $\rho_1+\rho_2$. The pure-factor terms in~\eqref{functional} give
$$
\mathcal{F}_X(\omega_1+\omega_2)
=E_1\int_{S_2}\mu_2+E_2\int_{S_1}\mu_1.
$$
The remaining terms have unequal holomorphic and antiholomorphic degrees on at least one surface and integrate to zero. Substituting~\eqref{surface_volume}, we obtain the product energy formula
\begin{equation}\label{product_energy}
\mathcal{F}_X(\omega_1+\omega_2)=V_2E_1+V_1E_2-2E_1E_2.
\end{equation}

\subsection{The full critical-point condition}

We retain the product background $(X,\eta)=(S_1\times S_2,\eta_1+\eta_2)$ and the surface data~\eqref{surface_metrics}--\eqref{surface_data}. In particular, the 2-form $\omega=\omega_1+\omega_2$ is the Hermitian product metric and $\rho=\rho_1+\rho_2$ is its canonical completion. We set $\Theta=\Theta_1+\Theta_2=\omega+\rho+\bar\rho$. 

The next lemma is an identity of top-degree forms on $X$ and permits arbitrary mixed components in the forms being varied.

\begin{lemma}\label{pointwise_identity}
Suppose the forms in~\eqref{surface_data} satisfy $\nu_i=\mu_i$ at every point of $S_i$ for each $i\in\{1,2\}$. Then, for every real $(1,1)$-form $\psi$ and every $(2,0)$-form $\sigma$ on $X$, it holds that
\begin{equation}\label{density_identity}
\operatorname{Re}\left(\sigma\wedge\bar\rho\wedge\omega^2\right)
+\rho\wedge\bar\rho\wedge \psi\wedge\omega=\frac1{12}(\psi+\sigma+\bar\sigma)\wedge\Theta^3.
\end{equation}
\end{lemma}

\begin{proof}
We set $\zeta=\psi+\sigma+\bar\sigma$, where $\psi$ has bidegree $(1,1)$ and $\sigma$ has bidegree $(2,0)$ as specified in the lemma. For the product decomposition $X=S_1\times S_2$, the degree on each factor counts the covectors coming from that factor. Every form occurring in $\rho$, $\bar\rho$ and $\omega$ has even degree on each factor. A mixed component of $\psi$ or $\sigma$ has odd degree on each factor, so its contribution to either side of~\eqref{density_identity} cannot have top degree on both surfaces and is zero.

For the components of $\psi$ and $\sigma$ supported on a single surface, expansion of the left-hand side of~\eqref{density_identity} and substitution of $\nu_i=\mu_i$ give
$$
\zeta\wedge(\Theta_1\wedge \mu_2+\Theta_2\wedge \mu_1).
$$
For example, a $(1,1)$-component of $\psi$ on $S_1$ pairs with $\omega_1\wedge \nu_2=\omega_1\wedge \mu_2$, while a $(2,0)$-component of $\sigma$ on $S_1$ pairs with $\bar\rho_1\wedge \mu_2$. The conjugate component supplies the corresponding conjugate term. The contributions supported on $S_2$ are obtained by exchanging the two surface factors.

The equality $\nu_i=\mu_i$ also gives $\Theta_i^2=4\mu_i$. Since a surface has real dimension four, $\Theta_i^3=0$. Expanding the cube of $\Theta_1+\Theta_2$, we obtain that
$$
\Theta^3=12(\Theta_1\wedge \mu_2+\Theta_2\wedge \mu_1).
$$
Substitution of this expression for $\Theta^3$ proves~\eqref{density_identity}.
\end{proof}

\begin{theorem}\label{critical_products}
Let $(S_1,\eta_1)$ and $(S_2,\eta_2)$ be K\"ahler surfaces, and let $\omega_i\in[\eta_i]_A$ be the positive Hermitian forms in~\eqref{surface_metrics}, with canonical completions $\rho_i=\partial v_i$. On the product background $(X,\eta)=(S_1\times S_2,\eta_1+\eta_2)$, the metric $\omega=\omega_1+\omega_2$ is a critical point of $\mathcal{F}_X$ in the full Aeppli class $[\eta]_A$ if and only if either $\omega_1$ is K\"ahler on $S_1$ and $\omega_2$ is K\"ahler on $S_2$, or it holds that
\begin{equation}\label{critical_densities}
\begin{cases}
2\rho_1\wedge\bar\rho_1=\omega_1^2,\\
2\rho_2\wedge\bar\rho_2=\omega_2^2.
\end{cases}
\end{equation}
Each equality in~\eqref{critical_densities} is an equality of top-degree forms at every point of the indicated surface.
\end{theorem}

\begin{proof}
Suppose first that the surface forms $\omega_1$ and $\omega_2$ are K\"ahler on $S_1$ and $S_2$, respectively. Then $\rho_1=\rho_2=0$, so $\mathcal{F}_X(\omega_1+\omega_2)=0$. Since the functional $\mathcal{F}_X$ is nonnegative on every positive representative of $[\eta]_A$, the product metric $\omega_1+\omega_2$ is an absolute minimizer and hence a critical point in that full Aeppli class. 

Suppose instead that the two equalities~\eqref{critical_densities} hold. Let $u$ be an arbitrary $(1,0)$-form on $X$, and set $\psi=\partial\bar u+\bar\partial u$. The family $\rho+t\partial u$ completes $\omega+t\psi$ and equals the canonical completion at $t=0$. Applying Lemma~\ref{pointwise_identity} with $\sigma=\partial u$ in~\eqref{first_variation}, we obtain that
\begin{equation}\label{stokes_variation}
(D\mathcal{F}_X)_\omega\psi
=\frac1{12}\int_Xd(u+\bar u)\wedge\Theta^3=0.
\end{equation}
The last equality follows from $d\Theta=0$ and Stokes' theorem. Since $u$ was arbitrary, the metric $\omega$ is critical in the full Aeppli class $[\eta]_A$.

Conversely, suppose that the product metric $\omega=\omega_1+\omega_2$ is critical for $\mathcal{F}_X$ in $[\eta]_A$ and that the first surface completion energy $E_1$ is positive. Let $\phi$ be an arbitrary smooth real function on $S_2$, and define the $(1,0)$-form $u$ on $X$ by $u(x,y)=\phi(y)v_1(x)$, where $x\in S_1$ and $y\in S_2$. In the derivative $\partial u$, the pure component supported on $S_1$ is $\phi\rho_1$, the pure component supported on $S_2$ is zero, and the remaining component is mixed. In the real variation $\psi=\partial\bar u+\bar\partial u$, the pure component supported on $S_1$ is $\phi(\partial\bar v_1+\bar\partial v_1)$, and the pure component supported on $S_2$ is zero. Every mixed component contributes zero to~\eqref{first_variation} by the separate-degree argument used in Lemma~\ref{pointwise_identity}.

To compute the remaining surface integral, replace $v_1$ by $(1+s)v_1$ in~\eqref{surface_metrics}. This replacement changes the completion $\rho_1$ to $(1+s)\rho_1$ and the energy $E_1$ to $(1+s)^2E_1$, while the number $V_1$ remains fixed. Differentiating~\eqref{surface_volume} at $s=0$ therefore gives
$$
\int_{S_1}(\partial\bar v_1+\bar\partial v_1)\wedge\omega_1=-2E_1.
$$
Now, it is readily seen that
\begin{equation}\label{localized_variation}
(D\mathcal{F}_X)_\omega\psi
=2E_1\int_{S_2}\phi(\mu_2-\nu_2).
\end{equation}
Full criticality of $\omega$ makes the left-hand side of~\eqref{localized_variation} zero for every choice of $\phi$. Since $E_1>0$, we conclude that $\mu_2=\nu_2$ pointwise on $S_2$. The positivity of $\mu_2$ implies $$E_2=\int_{S_2}\nu_2>0.$$ We may therefore repeat the same argument using an arbitrary real function on $S_1$ multiplied by the $(1,0)$-form $v_2$ on $S_2$. That variation gives $\mu_1=\nu_1$ pointwise on $S_1$.

If the initially nonzero energy is $E_2$, the preceding argument applies after exchanging the indices $1$ and $2$. If both completion energies $E_1$ and $E_2$ vanish, the equivalence following~\eqref{surface_volume} shows that $\omega_1$ is K\"ahler on $S_1$ and $\omega_2$ is K\"ahler on $S_2$. These cases exhaust the possibilities and prove the theorem.
\end{proof}

\section{Existence in Prescribed K\"ahler Classes}\label{section_existence}

For a complex surface $S$, we denote its canonical line bundle by $K_S=\Lambda^{2,0}T^*S$. The complex line bundle $K_S$ is smoothly trivial if it admits a nowhere-zero smooth section. Smooth triviality of $K_S$ is equivalent to the vanishing of the first Chern class $c_1(K_S)$ in $H^2(S,\mathbb Z)$. The integral cohomology group is essential here: vanishing of $c_1(K_S)$ only in $H^2(S,\mathbb R)$ is not the condition in the following theorem.

\begin{theorem}\label{existence}
Let $(S_1,\eta_1)$ and $(S_2,\eta_2)$ be compact K\"ahler surfaces, and equip $X=S_1\times S_2$ with the product K\"ahler form $\eta=\eta_1+\eta_2$. There exists a non-K\"ahler product Hermitian metric $\omega=\omega_1+\omega_2$, with $\omega_i\in[\eta_i]_A$ for each $i\in\{1,2\}$, which is a critical point of the Dinew--Popovici functional $\mathcal{F}_X$ on the full space of positive representatives of the Aeppli class $[\eta]_A$, if and only if
\begin{equation}\label{chern_condition}
c_1(K_{S_i})=0\in H^2(S_i,\mathbb Z)
\end{equation}
for each $i\in\{1,2\}$.
\end{theorem}

For the necessity in Theorem~\ref{existence}, Theorem~\ref{critical_products} gives $\rho_i\wedge\bar\rho_i=\omega_i^2/2$ for the canonical completion $\rho_i$ of each factor metric $\omega_i$. Since the volume form $\omega_i^2/2$ is positive everywhere on $S_i$, the form $\rho_i$ is a nowhere-zero smooth section of $K_{S_i}$. Thus, $K_{S_i}$ is smoothly trivial for each $i\in\{1,2\}$. We shall prove the converse by first constructing a nowhere-zero exact $(2,0)$-form on each surface and then prescribing the volume form of a Hermitian metric on that surface.

\begin{lemma}\label{exact_nonvanishing_form}
Let $(S,\eta)$ be a compact K\"ahler surface with $c_1(K_S)=0$ in $H^2(S,\mathbb Z)$. Then, there exists a nowhere-zero smooth $(2,0)$-form $\rho$ on $S$ which is $\partial$-exact, namely $\rho=\partial v$ for a smooth $(1,0)$-form $v$ on $S$.
\end{lemma}

\begin{proof}
We write $h^{2,0}(S)$ for the complex dimension of the space of holomorphic 2-forms on $S$. Suppose first that $h^{2,0}(S)=0$, and choose a nowhere-zero smooth section $\rho$ of the smoothly trivial bundle $K_S$. Since $S$ has complex dimension two, the $(3,0)$-form $\partial\rho$ vanishes. The K\"ahler Hodge decomposition for the background form $\eta$ identifies the $\partial$-harmonic forms of bidegree $(2,0)$ with the holomorphic 2-forms on $S$. Thus, there is no harmonic component of $\rho$, and the $\partial$-closed form $\rho$ is $\partial$-exact.

Suppose now that $h^{2,0}(S)>0$, and choose a nonzero holomorphic 2-form $\Omega$ on $S$. The zero divisor of $\Omega$ represents $c_1(K_S)$. A nonempty effective divisor on $S$ has strictly positive integral against the K\"ahler form $\eta$, whereas $c_1(K_S)=0$. Consequently, the zero divisor of $\Omega$ is empty. The form $\Omega$ therefore trivializes $K_S$ holomorphically. Every other holomorphic 2-form on $S$ is a holomorphic function times $\Omega$, and every holomorphic function on the compact connected complex manifold $S$ is constant. Hence, $h^{2,0}(S)=1$.

We normalize the positive measure proportional to $\Omega\wedge\bar\Omega$ to have total mass one, and denote the resulting probability measure on $S$ by $\mu$. Choose a nonconstant smooth real function $\phi$ on $S$ and consider the entire function
$$
M(z)=\int_S e^{z\phi}\,d\mu.
$$
We claim that $M$ has a zero in $\mathbb C$. Indeed, compactness of $S$ gives the estimate $$|M(z)|\leq e^{\|\phi\|_{\infty}|z|}.$$ If $M$ had no zeros, the Hadamard factorization would then imply that $M(z)=e^{az+b}$ for some constants $a,b\in\mathbb C$. Since $M(0)=1$, this exponential representation would give
$$
0=M''(0)-M'(0)^2
=\int_S \phi^2\,d\mu-\left(\int_S \phi\,d\mu\right)^2.
$$
The expression on the right is the variance of $\phi$ with respect to $\mu$, namely the integral of the squared difference between $\phi$ and its mean. This variance is strictly positive because $\phi$ is nonconstant and $\mu$ is positive on every nonempty open subset of $S$. We obtain a contradiction, so $M$ has a zero.

Let $c\in\mathbb C$ satisfy $M(c)=0$, and set $\rho=e^{c\phi}\Omega$. The form $\rho$ is nowhere zero, since neither the exponential factor nor $\Omega$ vanishes. Moreover, the surface $L^2(\eta)$ pairing gives
$$
\langle\rho,\Omega\rangle_{L^2(\eta)}
=\int_S e^{c\phi}\Omega\wedge\bar\Omega
=M(c)\int_S\Omega\wedge\bar\Omega=0.
$$
The form $\rho$ is $\partial$-closed by degree, and the $\partial$-harmonic space of bidegree $(2,0)$ for the K\"ahler form $\eta$ is spanned by $\Omega$. The K\"ahler Hodge decomposition therefore implies that $\rho$ is $\partial$-exact.
\end{proof}

\begin{proposition}\label{surface_density_existence}
Let $(S,\eta)$ be a compact K\"ahler surface whose canonical bundle $K_S$ is smoothly trivial. Then, there exists a non-K\"ahler Hermitian form $\omega\in[\eta]_A$ whose canonical completion $\rho$ satisfies $\omega^2/2=\rho\wedge\bar\rho$ pointwise on $S$.
\end{proposition}

\begin{proof}
For simplicity, we set $$V:=\frac{1}{2}\int_S\eta^2,$$ the total volume of the background K\"ahler metric $\eta$. By Lemma~\ref{exact_nonvanishing_form}, we may choose a nowhere-zero form $\rho=\partial v$ on $S$ and multiply $\rho$ and $v$ by the same positive constant so that
$$
\int_S\rho\wedge\bar\rho=\frac{V}{2}.
$$
Consider the real pluriclosed $(1,1)$-form $\alpha=\eta+\partial\bar v+\bar\partial v$ on $S$. We do not yet assert that $\alpha$ is positive. The real 2-form $\alpha+\rho+\bar\rho=\eta+d(v+\bar v)$ is closed and is cohomologous to $\eta$. By expanding the square of this closed form and applying Stokes' theorem, we obtain that
\begin{equation}\label{positivity_data}
\frac{1}{2}\displaystyle\int_S \alpha^2=\frac{V}{2}>0,\qquad
\displaystyle\int_S \alpha\wedge\eta=\int_S\eta^2>0,\qquad
\displaystyle\int_D \alpha=\int_D\eta>0
\end{equation}
for every irreducible complex curve $D\subset S$. In the third expression, integration over a singular curve means integration of the pulled-back form over the normalization of that curve. The pullbacks of $\rho$ and $\bar\rho$ to a complex curve vanish by degree. Restricting $\alpha+\rho+\bar\rho=\eta+d(v+\bar v)$ to the normalization of $D$ and applying Stokes' theorem therefore gives the third equality in~\eqref{positivity_data}.

By Buchdahl's positivity theorem \cite{Buchdahl}, the conditions~\eqref{positivity_data} imply that $\hat\omega=\alpha+dd^c\psi>0$ for some smooth real function $\psi$ on $S$. We next apply the Hermitian Monge--Amp\`ere theorem \cite{TosattiWeinkove} with background Hermitian form $\hat\omega$ and prescribed positive density $\rho\wedge\bar\rho$. There exist a smooth real function $\varphi$ on $S$ and a constant $c>0$ such that
$$
\frac{(\hat\omega+dd^c\varphi)^2}{2}=c\rho\wedge\bar\rho,
$$
with $\hat\omega+dd^c\varphi>0$. The background Hermitian form $\hat\omega$ is pluriclosed because $\alpha$ is pluriclosed. Stokes' theorem therefore shows that adding $dd^c\varphi$ to $\hat\omega$ does not change the integral of the square. The same argument applies when adding $dd^c\psi$ to $\alpha$, so we have that
$$
\frac{1}{2}\int_S(\hat\omega+dd^c\varphi)^2
=\frac{1}{2}\int_S\hat\omega^2
=\frac{1}{2}\int_S\alpha^2=\frac{V}{2}.
$$
Since the prescribed density $\rho\wedge\bar\rho$ also has integral $V/2$, integration of the Monge--Amp\`ere equation gives $c=1$.

We set $\omega=\hat\omega+dd^c\varphi$. The positive Hermitian form $\omega$ belongs to the Aeppli class $[\eta]_A$, and adding the two $dd^c$ terms has not changed the equation $\bar\partial\rho=-\partial \alpha$. Thus, we obtain that $\rho$ is a completion of $\omega$. Moreover, we have $\rho=\partial v$ is orthogonal to every holomorphic 2-form on $S$ in the metric-independent surface pairing, as in Lemma~\ref{surface_completion}. Hence, we conclude that $\rho$ is the canonical completion of $\omega$. If $\omega$ were K\"ahler, the equation $\bar\partial\rho=-\partial\omega$ would make $\rho$ holomorphic. Orthogonality of the canonical completion $\rho$ to all holomorphic 2-forms would then force $\rho=0$, contradicting the fact that $\rho$ is nowhere zero. This proves that $\omega$ is non-K\"ahler.
\end{proof}

\begin{proof}[Proof of Theorem~\ref{existence}]
The necessity of~\eqref{chern_condition} was proved immediately after the statement of Theorem~\ref{existence}. Conversely, suppose that~\eqref{chern_condition} holds for the two compact K\"ahler surfaces $(S_1,\eta_1)$ and $(S_2,\eta_2)$. Applying Proposition~\ref{surface_density_existence} to each surface gives a non-K\"ahler Hermitian form $\omega_i\in[\eta_i]_A$ whose canonical completion $\rho_i$ satisfies $\rho_i\wedge\bar\rho_i=\omega_i^2/2$. The product Hermitian form $\omega=\omega_1+\omega_2$ is therefore non-K\"ahler and satisfies~\eqref{critical_densities}. By Theorem~\ref{critical_products}, the form $\omega$ is a critical point of $\mathcal{F}_X$ on the full space of positive representatives of the Aeppli class $[\eta_1+\eta_2]_A$.
\end{proof}

K3 surfaces and complex 2-tori have trivial canonical bundles. Hence, for compact K\"ahler surfaces $(S_i,\eta_i)$ whose underlying complex surfaces are of these types, Theorem~\ref{existence} provides a non-K\"ahler full critical product in every prescribed class $[\eta_1+\eta_2]_A$. Conversely, if either integral class $c_1(K_{S_i})$ is nonzero, every full critical product with $\omega_i\in[\eta_i]_A$ is K\"ahler. In particular, this excludes non-K\"ahler full critical products with a $\mathbb{P}^2$ factor, because $c_1(K_{\mathbb{P}^2})$ is minus three times the hyperplane class. These conclusions concern product metrics. The construction in Proposition~\ref{surface_density_existence} does not require either surface to be projective.

\section{Torsion Flow and Static Instability}\label{section_flow}

Let $(X,\eta)$ be a compact K\"ahler manifold of complex dimension $n$. We keep the background K\"ahler form $\eta$ fixed and allow the Hermitian form $\omega$ to vary among the positive representatives of $[\eta]_A$. We denote the nonnegative Hodge Laplacian of the K\"ahler metric $\eta$ by $$\Delta_\eta=dd_\eta^*+d_\eta^*d,$$ where $d_\eta^*$ is the formal adjoint of $d$ in the $L^2(\eta)$ inner product. We write $\Lambda_\eta$ for contraction with $\eta$, equivalently, the adjoint of the Lefschetz operator.

For a time-dependent positive Hermitian form $\omega(t)$ on $X$, we consider the equation
\begin{equation}\label{torsion_flow}
\frac{\partial\omega}{\partial t}
=-\frac12\Delta_\eta\omega
+\frac12dd^c\left(\log(\omega^n/\eta^n)-\Lambda_\eta\omega\right)
-\lambda(\omega-\eta),
\end{equation}
where $\lambda>0$ is a fixed constant. The ratio $\omega^n/\eta^n$ denotes the positive function obtained by dividing the two volume forms pointwise. The diffusion coefficient in~\eqref{torsion_flow} is fixed at $1/2$. For a fixed positive Hermitian form $\omega\in[\eta]_A$, we denote the right-hand side of~\eqref{torsion_flow}, evaluated at $\omega$, by $\xi_\omega$. Thus, $\xi_\omega$ is a real $(1,1)$-form on $X$, and is the prescribed initial velocity of the equation at $\omega$.

The form $\xi_\omega$ is tangent to the Aeppli class $[\eta]_A$. Indeed, if $\omega-\eta=\partial\bar v+\bar\partial v$, the commutation of $\Delta_\eta$ with $\partial$ and $\bar\partial$, together with $\Delta_\eta\eta=0$, shows that $\Delta_\eta\omega$ is Aeppli-exact. The $dd^c$ term and the damping term in~\eqref{torsion_flow} are also Aeppli-exact. Consequently, the affine curve $\omega+s\xi_\omega$ stays in $[\eta]_A$, and remains positive for sufficiently small $|s|$. This affine curve is available even when no solution of the evolution equation has been constructed.

For a Hermitian form $\omega$, we call the real 3-form $d\omega$ the exterior torsion form of $\omega$. The $(2,1)$ component $\partial\omega$, together with the metric $\omega$, determines the torsion tensor of the Chern connection. Along any smooth positive solution of~\eqref{torsion_flow}, applying $d$ to the evolution equation gives
$$
\frac{\partial}{\partial t}d\omega
=-\frac12\Delta_\eta d\omega-\lambda d\omega.
$$
Here the $dd^c$ term vanishes after applying $d$, the background form $\eta$ is closed, and the Hodge Laplacian $\Delta_\eta$ commutes with $d$. In particular, the exterior torsion form of the evolving Hermitian metric satisfies a linear heat equation with damping on the fixed K\"ahler background.

\subsection{Surface dissipation}

We now take the compact K\"ahler surfaces $(S_1,\eta_1)$ and $(S_2,\eta_2)$ of Section~\ref{section_products}, and equip $X=S_1\times S_2$ with the background K\"ahler form $\eta=\eta_1+\eta_2$. We consider product Hermitian forms $\omega=\omega_1+\omega_2$ with $\omega_i\in[\eta_i]_A$. The symbols $\rho_i$, $V_i$ and $E_i$ continue to denote, respectively, the canonical completion of $\omega_i$, the total volume of $\eta_i$, and the surface energy $$\int_{S_i}\rho_i\wedge\bar\rho_i.$$

At a product form $\omega_1+\omega_2$, the velocity $\xi_\omega$ in~\eqref{torsion_flow} is the sum of the two surface velocities. Indeed, the logarithmic volume ratio and the contraction by the background form split into sums of functions on $S_1$ and $S_2$, and the product Hodge Laplacian restricts to $\Delta_{\eta_i}$ on forms pulled back from $S_i$. Thus, the right-hand side of~\eqref{torsion_flow} preserves the product form at the level of its velocity.

\begin{lemma}\label{surface_dissipation}
Let $\omega(t)=\omega_1(t)+\omega_2(t)$ be a smooth positive product solution of equation~\eqref{torsion_flow} on $(X,\eta_1+\eta_2)$, with $\omega_i(t)\in[\eta_i]_A$. For each $i\in\{1,2\}$, let $\rho_i(t)$ be the canonical completion of $\omega_i(t)$. Then, it holds that
\begin{equation}\label{completion_heat}
\frac{\partial\rho_i}{\partial t}
=-\frac12\Delta_{\eta_i}\rho_i-\lambda\rho_i.
\end{equation}
If we write $$E_i(t)=\int_{S_i}\rho_i(t)\wedge\bar\rho_i(t)$$ and set
\begin{equation}\label{dissipation_rates}
D_i(t)=\|d\omega_i(t)\|_{L^2(\eta_i)}^2+2\lambda E_i(t),
\end{equation}
then $E_i'(t)=-D_i(t)$. In particular, the norm in~\eqref{dissipation_rates} is taken with respect to the fixed background K\"ahler form $\eta_i$, rather than the evolving Hermitian form $\omega_i(t)$.

For a fixed positive product form $\omega=\omega_1+\omega_2$, the same identities for the initial derivatives of $\rho_i$ and $E_i$ hold along the affine curve $\omega(s)=\omega+s\xi_\omega$ at $s=0$.
\end{lemma}

\begin{proof}
Fix an index $i\in\{1,2\}$. The K\"ahler Laplacian $\Delta_{\eta_i}$ commutes with $\partial$ and $\bar\partial$. Applying $\partial$ to the surface equation for $\omega_i(t)$ eliminates the $dd^c$ term. Since $\partial\rho_i=0$ and $\bar\partial\rho_i=-\partial\omega_i$, the candidate derivative $-\Delta_{\eta_i}\rho_i/2-\lambda\rho_i$ consequently satisfies
$$
\bar\partial\left(\frac12\Delta_{\eta_i}\rho_i+\lambda\rho_i\right)
=\partial\frac{\partial\omega_i}{\partial t}.
$$
The candidate derivative is also $\partial$-closed. Thus, the candidate derivative satisfies the two equations obtained by differentiating the completion equations for $\rho_i(t)$.

To verify the normalization of this derivative, let $h$ be any holomorphic 2-form on $S_i$. The canonical completion $\rho_i(t)$ is orthogonal to $h$ in the surface pairing $$\int_{S_i}\rho_i(t)\wedge\bar h,$$ which is independent of the Hermitian metric on $S_i$. Differentiating this orthogonality shows that $\partial\rho_i/\partial t$ is orthogonal to $h$ in the fixed $L^2(\eta_i)$ inner product. Moreover, $h$ is harmonic for the K\"ahler metric $\eta_i$, so self-adjointness of $\Delta_{\eta_i}$ shows that $-\Delta_{\eta_i}\rho_i/2-\lambda\rho_i$ is also orthogonal to $h$. The difference between the actual derivative and the candidate derivative is a holomorphic 2-form orthogonal to all holomorphic 2-forms on $S_i$. That difference is therefore zero, proving~\eqref{completion_heat}.

Every $(2,0)$-form on the K\"ahler surface $(S_i,\eta_i)$ is self-dual under the complex-linear extension of the real Hodge star associated with $\eta_i$. For $\rho_i$, self-duality gives $\|d_{\eta_i}^*\rho_i\|_{L^2(\eta_i)}=\|d\rho_i\|_{L^2(\eta_i)}$. Hence, integration by parts gives
\begin{equation}\label{surface_hodge_identity}
\langle\Delta_{\eta_i}\rho_i,\rho_i\rangle_{L^2(\eta_i)}
=2\|d\rho_i\|_{L^2(\eta_i)}^2
=\|d\omega_i\|_{L^2(\eta_i)}^2.
\end{equation}
For the last equality, the completion equations imply $d\rho_i=-\partial\omega_i$. The $(2,1)$-form $\partial\omega_i$ and the conjugate $(1,2)$-form $\bar\partial\omega_i$ have equal norms and are orthogonal in $L^2(\eta_i)$, so the squared norm of their sum $d\omega_i$ is twice the squared norm of $d\rho_i$.

The surface identity $E_i=\|\rho_i\|_{L^2(\eta_i)}^2$ allows us to differentiate the energy using only the fixed background pairing. By~\eqref{completion_heat} and~\eqref{surface_hodge_identity}, we obtain that
$$
E_i'=-\langle\Delta_{\eta_i}\rho_i,\rho_i\rangle_{L^2(\eta_i)}-2\lambda E_i=-D_i.
$$
The preceding argument uses the evolution equation only to specify the initial derivative of $\omega_i$. For the affine curve through $\omega$ with initial velocity $\xi_\omega$, the initial derivative of each factor is the same surface velocity. The formulas for the derivatives of the canonical completions and the surface energies therefore hold at $s=0$ along that affine curve as well.
\end{proof}

\subsection{The negative Hessian formula}

\begin{theorem}\label{flow_instability}
Let $(S_1,\eta_1)$ and $(S_2,\eta_2)$ be compact K\"ahler surfaces, and set $X=S_1\times S_2$ and $\eta=\eta_1+\eta_2$. Let $\omega=\omega_1+\omega_2$, with $\omega_i\in[\eta_i]_A$, be a non-K\"ahler critical point of the Dinew--Popovici functional $\mathcal{F}_X$ on the full space of positive representatives of the Aeppli class $[\eta]_A$. For this fixed metric $\omega$, let $D_i$ be the number in~\eqref{dissipation_rates} evaluated at $\omega_i$, and let $\xi_\omega$ be the velocity in~\eqref{torsion_flow} with background form $\eta$. Then, the Hessian of $\mathcal{F}_X$ at $\omega$ satisfies
\begin{equation}\label{negative_hessian}
H_\omega(\xi_\omega,\xi_\omega)=-4D_1D_2<0.
\end{equation}
In particular, the same Hessian direction satisfies the quantitative estimate
\begin{equation}\label{quantitative_hessian}
H_\omega(\xi_\omega,\xi_\omega)\leq-8\lambda^2\mathcal{F}_X(\omega).
\end{equation}
\end{theorem}

\begin{proof}
Theorem~\ref{critical_products} gives $\rho_i\wedge\bar\rho_i=\omega_i^2/2$ for the canonical completion of each surface metric $\omega_i$. Integrating these two pointwise equalities and using~\eqref{surface_volume}, we obtain that $E_i=V_i/2$.

Consider the affine product curve $\omega(s)=\omega+s\xi_\omega$, which passes through $\omega$ at $s=0$, remains in $[\eta]_A$, and is positive for sufficiently small $|s|$. Denote the two surface energies along this curve by $E_1(s)$ and $E_2(s)$. The constants $V_1$ and $V_2$ are fixed by the background forms $\eta_1$ and $\eta_2$. Differentiating the product energy identity~\eqref{product_energy} twice along this curve gives
$$
\frac{d^2}{ds^2}\mathcal{F}_X(\omega(s))
=(V_2-2E_2(s))E_1''(s)+(V_1-2E_1(s))E_2''(s)-4E_1'(s)E_2'(s).
$$
At $s=0$, the coefficients $V_2-2E_2(0)$ and $V_1-2E_1(0)$ vanish. Lemma~\ref{surface_dissipation} gives $E_i'(0)=-D_i$ for each factor. Since the curve is affine with velocity $\xi_\omega$, the second derivative at $s=0$ is $H_\omega(\xi_\omega,\xi_\omega)$, proving~\eqref{negative_hessian}.

For the quantitative bound, we have $D_i\geq2\lambda E_i=\lambda V_i$. The product energy identity at the critical metric gives $\mathcal{F}_X(\omega)=V_1V_2/2$. Substitution of all these inequalities into~\eqref{negative_hessian} yields~\eqref{quantitative_hessian}.
\end{proof}

The strict negativity in~\eqref{negative_hessian} also holds if the damping parameter is set to $\lambda=0$. Indeed, Theorem~\ref{critical_products} implies that each surface metric $\omega_i$ is non-K\"ahler, so neither exterior torsion form $d\omega_i$ is identically zero. Consequently, each number $D_i=\|d\omega_i\|_{L^2(\eta_i)}^2$ is positive. Thus, the negative second variation is not caused solely by the damping term. The proof of Theorem~\ref{flow_instability} is a static calculation along the explicitly defined affine curve $\omega+s\xi_\omega$; no existence theorem for a solution of the evolution equation is needed.

On the product K\"ahler background $(X,\eta_1+\eta_2)$, the local minima of $\mathcal{F}_X$ restricted to product metrics $\omega_1+\omega_2$, with $\omega_i\in[\eta_i]_A$, are precisely the K\"ahler product metrics. Indeed, independently replacing $v_i$ by $(1+s_i)v_i$ varies $E_i$ as $(1+s_i)^2E_i$. Formula~\eqref{product_energy} shows that a restricted critical product either has $E_1=E_2=0$ or satisfies $E_i=V_i/2$ for both indices. In the latter case, the mixed second derivative of the product energy as a function of $(E_1,E_2)$ is $-2$, and its two pure second derivatives vanish. Since the two positive energies can be varied independently, this Hessian has a positive and a negative direction among product variations. If $E_1=E_2=0$, both factors are K\"ahler and $\mathcal{F}_X=0$, which is an absolute minimum.

The construction also gives non-K\"ahler critical metrics in every complex dimension at least four. Let $\omega$ be a non-K\"ahler critical product metric in Theorem~\ref{flow_instability}, and let $(Y,\theta)$ be a compact K\"ahler manifold of complex dimension $k\geq1$. On the product K\"ahler background $(X\times Y,\eta+\theta)$, the form $\omega+\theta$ is critical for $\mathcal{F}_{X\times Y}$ in the full Aeppli class $[\eta+\theta]_A$. We write
$$
V(\theta)=\frac{1}{k!}\int_Y\theta^k
$$
for the total volume of $(Y,\theta)$. Expansion of~\eqref{functional} using the canonical completion described below gives
$$
\mathcal{F}_{X\times Y}(\omega+\theta)=V(\theta)\mathcal{F}_X(\omega).
$$
Indeed, the pullback of $\rho_\omega$ is the canonical completion of $\omega+\theta$: orthogonality to holomorphic 2-forms follows from canonicality on $X$ and the orthogonality of the product types. For an arbitrary smooth $(1,0)$-form $u$ on $X\times Y$, we decompose $u$ into components with covectors on $X$ and $Y$. We use $\rho_\omega+t\partial u$ when differentiating the energy along the variation $\partial\bar u+\bar\partial u$. The mixed components contribute zero to the first variation by their separate bidegrees. The components supported on $X$ give the first variation on each fibre $X\times\{y\}$, which vanishes by full criticality on $X$. The components supported on $Y$ give the variation of the $Y$ volume, weighted by the energy density on $X$. Their fibre integrals vanish by Stokes' theorem, since $d\theta=0$. This proves full criticality even when the coefficients of $u$ depend on both factors. The flow velocity at $\omega+\theta$, with background $\eta+\theta$, is the pullback of $\xi_\omega$. The same energy identity holds along the corresponding affine variation, so Theorem~\ref{flow_instability} gives
$$
H_{\omega+\theta}(\xi_\omega,\xi_\omega)=-4V(\theta)D_1D_2<0.
$$

The velocity $\xi_\omega$ thus supplies a distinguished negative Hessian direction at every non-K\"ahler critical product metric, with the choice of direction determined by the fixed K\"ahler background and the damping parameter. The Hessian value in that direction is expressed exactly in terms of the two surface torsion quantities $D_1$ and $D_2$. However, the nonlinear $dd^c$ correction in~\eqref{torsion_flow} does not enter the completion equation~\eqref{completion_heat}. Replacing that correction by any other $dd^c$ correction which preserves the product form leaves the surface completion derivative unchanged, provided that the heat and damping terms are kept fixed. Consequently, the negative Hessian identity~\eqref{negative_hessian} does not distinguish the particular nonlinear correction in~\eqref{torsion_flow} from all other product-preserving $dd^c$ corrections.

\section{Potential Variations on K\"ahler 4-folds}\label{section_potentials}

Throughout this section, $(X,\eta)$ is a compact K\"ahler manifold of complex dimension four, and $\omega$ is a Hermitian symplectic form on the complex manifold $X$. We do not assume that the Aeppli class $[\omega]_A$ contains the background K\"ahler form $\eta$. We consider the Dinew--Popovici functional $\mathcal{F}_X$ on the positive representatives of $[\omega]_A$, write $\rho=\rho_\omega$ for the canonical completion of $\omega$, and also set
$$
V(\omega)=\frac{1}{4!}\int_X\omega^4
$$
for the total volume of the Hermitian metric $\omega$.

For a real smooth function $\phi$ on $X$, we call $\omega_\phi=\omega+dd^c\phi$ a positive potential variation of $\omega$ if $\omega_\phi>0$. The forms $\omega_\phi$ and $\omega$ have the same Aeppli class and the same exterior torsion form, since $ddd^c\phi=0$. In particular, every positive potential variation of $\omega$ lies in the domain of $\mathcal{F}_X$ under consideration. When differentiating at $\omega$, every real smooth function $\phi$ is allowed: the positivity of $\omega+tdd^c\phi$ holds for all sufficiently small real $t$.

\subsection{The completion and total volume}

As usual, a Hermitian form $\omega$ on an $n$-dimensional complex manifold is called Gauduchon if $dd^c\omega^{n-1}=0$. We say that $\omega$ is stationary under potential variations if the derivative of $\mathcal{F}_X(\omega+tdd^c\phi)$ at $t=0$ vanishes for every real smooth function $\phi$ on $X$.

\begin{theorem}\label{potential_volume}
Let $\omega$ be a Hermitian symplectic form on the compact K\"ahler 4-fold $(X,\eta)$. For every positive potential variation $\omega_\phi=\omega+dd^c\phi$, the canonical completion of $\omega_\phi$ equals the canonical completion of $\omega$, and it holds that
\begin{equation}\label{energy_volume_identity}
\mathcal{F}_X(\omega_\phi)+V(\omega_\phi)=\mathcal{F}_X(\omega)+V(\omega).
\end{equation}
Moreover, the form $\omega$ is stationary for $\mathcal{F}_X$ under all potential variations if and only if $dd^c\omega^3=0$. Consequently, every critical point of $\mathcal{F}_X$ in the full space of positive representatives of $[\omega]_A$ is Gauduchon.
\end{theorem}

\begin{proof}
Let $\zeta$ be an arbitrary completion of $\omega$, and set $\Theta_\zeta=\omega+\zeta+\bar\zeta$. The real 2-form $\Theta_\zeta$ is closed by the definition of a completion. Expanding the fourth exterior power of $\Theta_\zeta$ and retaining the terms of bidegree $(4,4)$, we obtain that
\begin{equation}\label{completed_volume}
\frac{1}{4!}\int_X\Theta_\zeta^4
=V(\omega)+\frac{1}{2}\int_X\zeta\wedge\bar\zeta\wedge\omega^2
+\frac14\int_X\zeta^2\wedge\bar\zeta^{2}.
\end{equation}
The middle term on the right-hand side of~\eqref{completed_volume} is the squared $L^2(\omega)$ norm of the chosen completion $\zeta$.

The forms $\omega_\phi$ and $\omega$ have exactly the same completions. Indeed, the completion equations depend on the Hermitian form only through $\partial\omega$, and $\partial\omega_\phi=\partial\omega$. For the same chosen completion $\zeta$, the closed real form associated with $\omega_\phi$ is $\Theta_\zeta+dd^c\phi$. This form represents the same de Rham cohomology class as $\Theta_\zeta$. Hence, replacing $\omega$ by $\omega_\phi$ in~\eqref{completed_volume} leaves the left-hand side unchanged. The last term in~\eqref{completed_volume} also remains unchanged, since the form $\zeta$ has been kept fixed. Subtracting the two completed-volume identities, we conclude that
$$
\|\zeta\|_{L^2(\omega_\phi)}^2
-\|\zeta\|_{L^2(\omega)}^2
=V(\omega)-V(\omega_\phi).
$$
The right-hand side of this equality is independent of the chosen completion $\zeta$. Thus, the two squared norms, considered as functions on the common space of completions, differ by a constant. The unique completion minimizing the $L^2(\omega)$ norm therefore also minimizes the $L^2(\omega_\phi)$ norm. This proves that $\rho_{\omega_\phi}=\rho_\omega$, and evaluating the preceding equality at this common canonical completion proves~\eqref{energy_volume_identity}.

For an arbitrary real smooth function $\phi$, we now apply~\eqref{energy_volume_identity} to the curve $\omega+tdd^c\phi$. Differentiating at $t=0$ and integrating by parts, we obtain that
\begin{equation}\label{potential_first_variation}
\left.\frac{d}{dt}\mathcal{F}_X(\omega+tdd^c\phi)\right|_{t=0}
=-\frac16\int_X\omega^3\wedge dd^c\phi
=-\frac16\int_X\phi\,dd^c\omega^3.
\end{equation}
Since $\phi$ is arbitrary, the derivative in~\eqref{potential_first_variation} vanishes for every potential direction precisely when $dd^c\omega^3=0$. Every critical point in the full Aeppli class is stationary under these particular variations, which proves the final assertion.
\end{proof}

For each chosen completion $\zeta$, the proof preserves the de Rham class of the closed real form $\omega+\zeta+\bar\zeta$ along the potential variation. The argument neither requires this de Rham class to be independent of $\zeta$ nor identifies the class of $\omega+\rho_\omega+\bar\rho_\omega$ with the background class $[\eta]$.

\subsection{The potential Hessian}

We consider the restriction of the affine Hessian $H_\omega$ of $\mathcal{F}_X$ to the real vector space of potential directions $dd^c\phi$, where $\phi$ ranges over the real smooth functions on $X$. We call this restricted quadratic form the potential Hessian at $\omega$. A real quadratic form has infinite positive index if it is positive definite on subspaces of arbitrarily large finite dimension. Infinite negative index means that the quadratic form is negative definite on subspaces of arbitrarily large finite dimension. In the following theorem, these subspaces consist of potential directions at the fixed Hermitian form $\omega$.

\begin{theorem}\label{potential_index}
Let $\omega$ be a Hermitian symplectic form on the compact K\"ahler 4-fold $(X,\eta)$. If $dd^c\omega^2\ne0$, then the potential Hessian of $\mathcal{F}_X$ at $\omega$ has infinite positive and negative index; if $dd^c\omega^2=0$, then $\mathcal{F}_X(\omega_\phi)=\mathcal{F}_X(\omega)$ for every positive potential variation $\omega_\phi=\omega+dd^c\phi$.
\end{theorem}

\begin{proof}
Fix a real smooth function $\phi$ on $X$. Since the Hermitian symplectic form $\omega$ is pluriclosed, Stokes' theorem gives
$$
\displaystyle\int_X\omega\wedge(dd^c\phi)^3=0,\qquad
\displaystyle\int_X(dd^c\phi)^4=0.
$$
These are precisely the coefficients, up to fixed nonzero constants, of the cubic and quartic terms in $V(\omega+tdd^c\phi)$. Thus, the total volume $V(\omega+tdd^c\phi)$ is a polynomial of degree at most two in $t$. By~\eqref{energy_volume_identity}, the energy $\mathcal{F}_X(\omega+tdd^c\phi)$ has the same degree bound whenever the form $\omega+tdd^c\phi$ is positive. Differentiating the energy-volume identity twice at $t=0$, we obtain that
\begin{equation}\label{potential_hessian}
\begin{split}
H_\omega(dd^c\phi,dd^c\phi)
&=-\frac12\int_X\omega^2\wedge(dd^c\phi)^2\\
&=\frac12\int_Xdd^c\omega^2\wedge d\phi\wedge d^c\phi.
\end{split}
\end{equation}
For the second equality in~\eqref{potential_hessian}, we integrate by parts and use $d\phi\wedge d^c\phi=2i\partial\phi\wedge\bar\partial\phi$. The left-hand side denotes the affine second derivative defined in Section~\ref{section_products}; the equality does not require $\omega$ to be a critical point of $\mathcal{F}_X$.

We set $\Psi=dd^c\omega^2$, which is a real $(3,3)$-form on $X$. Applying Stokes' theorem and using $dd^c\omega=0$, we obtain that
$$
\int_X\Psi\wedge\omega=\int_X\omega^2\wedge dd^c\omega=0.
$$
Suppose first that $\Psi\ne0$. At each point $p\in X$, division of $\Psi\wedge i\zeta\wedge\bar\zeta$ by the positive volume form $\omega^4/4!$ defines a Hermitian quadratic form in the covector $\zeta\in T_p^*X^{1,0}$. These pointwise quadratic forms must take a strictly positive value somewhere on $X$ and a strictly negative value somewhere on $X$.

To prove this assertion, suppose that the pointwise quadratic forms were nonnegative for every $p$ and every $\zeta$. The traces of these pointwise quadratic forms, computed in $\omega$-unitary coframes, are positive constant multiples of the densities defined by $\Psi\wedge\omega$. Since $\Psi$ is not identically zero, the associated Hermitian quadratic form is nonzero at some point. Nonnegativity then makes the trace strictly positive at that point, and hence on a nonempty open set. This contradicts $$\int_X\Psi\wedge\omega=0.$$ The possibility that all the pointwise quadratic forms are nonpositive is excluded by the same argument applied to $-\Psi$. This proves the assertion about the signature.

Choose either of the two strict signs. At a point and a covector realizing the chosen sign, a complex-linear change of holomorphic coordinates allows us to take that covector to be a multiple of $dz_1$. By shrinking the coordinate neighbourhood, we obtain a coordinate ball on which $\Psi\wedge dx\wedge d^cx$ has the chosen strict sign, where $x=\operatorname{Re}z_1$. Let $a$ be a nonzero real smooth function compactly supported in this ball, and set $\phi_N=a\cos(Nx)$ for every positive integer $N$, extending $\phi_N$ by zero outside the ball. Formula~\eqref{potential_hessian} gives
$$
H_\omega(dd^c\phi_N,dd^c\phi_N)
=\frac{N^2}{2}\int_Xa^2\sin^2(Nx)\Psi\wedge dx\wedge d^cx+O(N).
$$
Here the bound implicit in $O(N)$ may depend on $\omega$, the chosen coordinates and the function $a$, but not on $N$. The identity $\sin^2(Nx)=(1-\cos(2Nx))/2$, followed by integration by parts in the coordinate $x$, shows that the integral containing $\sin^2(Nx)$ converges to
$$
\frac12\int_Xa^2\Psi\wedge dx\wedge d^cx.
$$
This limit is nonzero and has the chosen sign. Consequently, the potential Hessian evaluated at $dd^c\phi_N$ has the chosen sign for every sufficiently large $N$.

For any positive integer $k$, choose $k$ disjoint smaller coordinate balls inside a region with the chosen strict sign. The preceding construction produces one potential direction on each smaller ball with a nonzero Hessian value of that sign. The resulting directions have disjoint supports. By polarization of the last expression in~\eqref{potential_hessian}, the Hessian pairing of any two different directions is therefore zero. Each direction is nonzero, since its Hessian value is nonzero, and the disjoint supports make the $k$ directions linearly independent. The potential Hessian is thus definite with the chosen sign on the $k$-dimensional span of these potential directions. Repeating the construction in a region with the opposite sign proves that the potential Hessian has infinite positive and negative index. All these are admissible infinitesimal variations: on each of the finite-dimensional spans just constructed, the form obtained by adding a sufficiently small variation to $\omega$ remains positive.

Suppose now that $\Psi=0$, that is, $dd^c\omega^2=0$. The pluriclosedness of $\omega$ gives the identities
\begin{equation}\label{power_identities}
\begin{cases}
dd^c\omega^2=4i\partial\omega\wedge\bar\partial\omega,\\
dd^c\omega^3=3\omega\wedge dd^c\omega^2=0.
\end{cases}
\end{equation}
For every real smooth function $\phi$, the first derivative in~\eqref{potential_first_variation} and the second derivative in~\eqref{potential_hessian} therefore vanish. Since $\mathcal{F}_X(\omega+tdd^c\phi)$ is a polynomial of degree at most two on its interval of positivity, this polynomial is constant. If $\omega_\phi>0$, then every form on the segment from $\omega$ to $\omega_\phi$ is positive, because the cone of positive $(1,1)$-forms is convex. Taking $t=1$ proves $\mathcal{F}_X(\omega_\phi)=\mathcal{F}_X(\omega)$ for every positive potential variation, as required.
\end{proof}

\begin{corollary}\label{astheno_minimum}
Let $(X,\eta)$ be a compact K\"ahler 4-fold, and let $\omega$ be a local minimizer of the Dinew--Popovici functional $\mathcal{F}_X$ on the full space of positive representatives of the Hermitian symplectic Aeppli class $[\omega]_A$. Then, it holds that $dd^c\omega^2=0$.
\end{corollary}

\begin{proof}
Since every potential direction $dd^c\phi$ is tangent to the Aeppli class $[\omega]_A$, a local minimizer $\omega$ is stationary under potential variations and satisfies $H_\omega(dd^c\phi,dd^c\phi)\geq0$ for every real smooth function $\phi$. If $dd^c\omega^2\ne0$, Theorem~\ref{potential_index} produces potential directions with strictly negative second variation, which is a contradiction.
\end{proof}

A Hermitian form $\omega$ on an $n$-dimensional complex manifold is said to be astheno-K\"ahler if $dd^c\omega^{n-2}=0$. Thus, Corollary~\ref{astheno_minimum} says that every local minimizer of $\mathcal{F}_X$ on a compact K\"ahler 4-fold $(X,\eta)$ is astheno-K\"ahler. The condition $dd^c\omega^2=0$ obtained in this corollary does not by itself assert the K\"ahler condition $d\omega=0$.

\begin{corollary}\label{product_infinite_index}
Let $(S_1,\eta_1)$ and $(S_2,\eta_2)$ be compact K\"ahler surfaces, and let $\omega=\omega_1+\omega_2$ be a non-K\"ahler critical product metric for $\mathcal{F}_X$ as in Theorem~\ref{critical_products}, where $X=S_1\times S_2$ has background K\"ahler form $\eta=\eta_1+\eta_2$. Then, the Hessian of $\mathcal{F}_X$ at $\omega$ has infinite positive and negative index even when restricted to potential variations. Every metric in each of these potential variations has exterior torsion form equal to $d\omega$.
\end{corollary}

\begin{proof}
Theorem~\ref{critical_products} implies that neither surface metric $\omega_1$ nor surface metric $\omega_2$ is K\"ahler. Since the forms $\omega_i$ are real, the form $\partial\omega_i$ is nonzero at some point of $S_i$ for each $i\in\{1,2\}$. The products $\partial\omega_i\wedge\bar\partial\omega_i$ vanish because each surface $S_i$ has complex dimension two. The two cross products $\partial\omega_1\wedge\bar\partial\omega_2$ and $\partial\omega_2\wedge\bar\partial\omega_1$ have separate bidegrees $((2,1),(1,2))$ and $((1,2),(2,1))$ respectively, where the two entries record the bidegrees on $S_1$ and $S_2$.

Choose points $p_i\in S_i$ with $\partial\omega_i|_{p_i}\ne0$ for $i\in\{1,2\}$. At the product point $(p_1,p_2)$, each of the two cross products is nonzero. The cross products cannot cancel, since their separate bidegrees are different. It follows that $\partial\omega\wedge\bar\partial\omega$ is nonzero at $(p_1,p_2)$. The first identity in~\eqref{power_identities} therefore gives $dd^c\omega^2\ne0$, and Theorem~\ref{potential_index} proves the asserted infinite indices. Finally, every potential variation $\omega_\phi=\omega+dd^c\phi$ satisfies $d\omega_\phi=d\omega$, since $ddd^c\phi=0$.
\end{proof}

The tensor kept fixed in Corollary~\ref{product_infinite_index} is the real 3-form $d\omega$. When the Hermitian metric changes from $\omega$ to $\omega_\phi$, the Chern torsion tensor with an index raised by that metric need not remain fixed. The corollary therefore concerns the positive and negative second variations of $\mathcal{F}_X$ at fixed exterior torsion form; it does not impose constancy of the raised Chern torsion tensor or of its norm.

\section{Concluding Remarks}

Theorem~\ref{critical_products} classifies the product metrics which are critical for $\mathcal{F}_X$ under every Aeppli variation. Together with Theorem~\ref{existence}, this gives non-K\"ahler full critical products in the prescribed class $[\eta_1+\eta_2]_A$ exactly when the canonical bundles of both compact K\"ahler surfaces $(S_i,\eta_i)$ are smoothly trivial. These examples give a negative answer to the higher-dimensional critical-point question in~\cite[Question~1.5]{Soheil}.

The potential Hessian gives the main general restriction on local minimizers. On a compact K\"ahler 4-fold $(X,\eta)$, every local minimizer $\omega$ of $\mathcal{F}_X$ satisfies $dd^c\omega^2=0$ by Corollary~\ref{astheno_minimum}. In contrast, every non-K\"ahler critical product constructed here has infinite positive and negative Hessian indices, even under variations preserving $d\omega$, by Corollary~\ref{product_infinite_index}.

The torsion flow selects the negative direction $\xi_\omega$ and gives the exact Hessian value~\eqref{negative_hessian}. This static calculation uses only the initial velocity. The separate convergence theorem in Appendix~\ref{appendix_flow} applies to every smooth positive Hermitian initial form when the fixed K\"ahler background $(X,\eta)$ has nonnegative holomorphic bisectional curvature. The curvature hypothesis is unnecessary for the static results in the main text.

\appendix
\section{Existence and Convergence of the Torsion Flow}\label{appendix_flow}

Throughout this appendix, we fix once and for all a compact K\"ahler manifold $(X,\eta)$ of complex dimension $n\geq2$, and $\lambda>0$ is a given constant. We consider the torsion flow~\eqref{torsion_flow}. In this appendix the initial Hermitian form $\omega_0$ may be any smooth positive real $(1,1)$-form on $X$. In particular, we do not require $\omega_0$ to be closed, pluriclosed, or a representative of $[\eta]_A$. The flow equation is defined on all such initial forms, independently of the domain of the functional $\mathcal{F}_X$ in the main text.

All Hodge Laplacians, formal adjoints and norms in this appendix are taken with respect to the fixed K\"ahler form $\eta$, unless another metric is written explicitly. For example, by $C^r$ norm we mean the maximum of the supremum norms of the covariant derivatives of orders $0,\dots,r$, taken with respect to the metric $\eta$ and its Levi--Civita connection. The expression $e^{-t\Delta_\eta/2}$ denotes the Hodge heat solution operator of $\Delta_\eta$, so that for a smooth initial differential form $\alpha_0$, the evolution $\alpha(t):=e^{-t\Delta_\eta/2}\alpha_0$ solves $2\partial_t \alpha=-\Delta_\eta\alpha$ with $\alpha(0)=\alpha_0$.

\subsection{The scalar equation and short-time existence}

The K\"ahler identity $d_\eta^*=[\Lambda_\eta,d^c]$ gives
\begin{equation}\label{appendix_torsion_operator}
-\Delta_\eta\omega-dd^c\Lambda_\eta\omega
=-d_\eta^*d\omega-d\Lambda_\eta d^c\omega.
\end{equation}
Thus, the sum of the Hodge diffusion and trace terms in~\eqref{torsion_flow} is a first-order differential expression in the exterior torsion form $d\omega$. Here $d\omega$ determines $d^c\omega$ because $\omega$ is real and has bidegree $(1,1)$. The expression~\eqref{appendix_torsion_operator} vanishes for every closed real $(1,1)$-form. The logarithmic and damping terms in the full flow need not vanish at a K\"ahler form other than the prescribed background $\eta$.

For the initial form $\omega_0>0$, define the reference $(1,1)$-form by
\begin{equation}\label{appendix_reference}
\hat{\omega}(t)=\eta+e^{-\lambda t}e^{-t\Delta_\eta/2}(\omega_0-\eta).
\end{equation}
The form $\hat{\omega}(t)$ is determined by a linear equation and is not initially assumed to be positive for all time. It satisfies $\hat{\omega}(0)=\omega_0$ and
$$
\frac{\partial \hat{\omega}}{\partial t}=-\frac12\Delta_\eta \hat{\omega}-\lambda(\hat{\omega}-\eta).
$$

\begin{proposition}\label{appendix_local_existence}
Every smooth positive Hermitian form $\omega_0$ on $(X,\eta)$ is the initial value of a unique maximal smooth positive solution of~\eqref{torsion_flow}. On the interval of existence, that solution has the representation
\begin{equation}\label{appendix_scalar_equation}
\begin{cases}
\omega(t)=\hat{\omega}(t)+dd^c u(t),\\[3pt]
\displaystyle\frac{\partial u}{\partial t}
=\frac12\log\frac{(\hat{\omega}(t)+dd^c u)^n}{\eta^n}
+\frac12\Lambda_\eta(\eta-\hat{\omega}(t))-\lambda u,\\[5pt]
u(0)=0,
\end{cases}
\end{equation}
where $u$ is a smooth real function on $X$ depending on $t$, and $\hat{\omega}(t)$ is the form in~\eqref{appendix_reference}.
\end{proposition}

\begin{proof}
The scalar equation in~\eqref{appendix_scalar_equation} is strictly parabolic while $$\omega=\hat{\omega}+dd^c u>0,$$ since its linearization in a real function $v$ has principal part $\Lambda_\omega dd^c v/2$. Standard local parabolic theory gives a unique short-time solution with $u(0)=0$.

The K\"ahler identity $$\Delta_\eta dd^c u=-dd^c\Lambda_\eta dd^c u$$ shows that applying $dd^c$ to the scalar equation and adding the equation for $\hat{\omega}$ gives~\eqref{torsion_flow}. Conversely, the difference between any positive solution $\omega(t)$ and $\hat{\omega}(t)$ has zero initial value and solves a Hodge heat equation with $dd^c$-exact forcing. Commutation of the heat operator with $dd^c$ therefore gives $\omega-\hat{\omega}=dd^c u$. Adjusting the additive function of time in $u$ gives the stated scalar equation and $u(0)=0$. Scalar uniqueness proves uniqueness for the form equation, and local continuation gives the maximal positive solution.
\end{proof}

\subsection{Evolution of the torsion}

\begin{proposition}\label{appendix_torsion}
For every smooth positive solution of~\eqref{torsion_flow} with initial form $\omega_0$, it holds that
\begin{equation}\label{appendix_torsion_heat}
d\omega(t)=e^{-\lambda t}e^{-t\Delta_\eta/2}d\omega_0.
\end{equation}
In particular, K\"ahler initial forms remain K\"ahler.
\end{proposition}

\begin{proof}
Applying $d$ to~\eqref{torsion_flow} gives $$\partial_t d\omega=-\frac{1}{2}\Delta_\eta d\omega-\lambda d\omega,$$ which proves~\eqref{appendix_torsion_heat} and preservation of closedness.
\end{proof}

\subsection{A positive reference form and global convergence}

\begin{theorem}\label{appendix_reference_convergence}
Let $\omega_0$ be any smooth positive Hermitian form on a compact K\"ahler manifold $(X,\eta)$ of dimension $n$. Suppose that the form $\hat{\omega}(t)$ in~\eqref{appendix_reference} satisfies $\hat{\omega}(t)\geq c\eta$ for every $t\geq0$, where $c>0$ is independent of $t$. Then, the solution of~\eqref{torsion_flow} with initial form $\omega_0$ exists for all $t\geq0$, and there is a constant $C>0$ such that
\begin{equation}\label{appendix_uniform_metric}
C^{-1}\eta\leq\omega(t)\leq C\eta\qquad(t\geq0).
\end{equation}
For every nonnegative integer $r$ and every $0<\gamma<\lambda$, there is a constant $C>0$, depending on $r$ and $\gamma$ as well as the fixed data, such that the estimate 
\begin{equation}\label{appendix_convergence_rate}
\|\omega(t)-\eta\|_{C^r}\leq Ce^{-\gamma t}
\end{equation}
of $C^r$ norms holds for all $t\geq1$.
\end{theorem}

\begin{proof}
The Hodge heat formula~\eqref{appendix_reference} gives uniform bounds for all space and time derivatives of $\hat{\omega}(t)$, and $\hat{\omega}(t)-\eta=O(e^{-\lambda t})$ in every smooth norm. Together with the assumed lower bound, these estimates imply that $\hat{\omega}$ is uniformly equivalent to $\eta$. We use the scalar representation~\eqref{appendix_scalar_equation} and put $f=-\Lambda_\eta(\hat{\omega}-\eta)/2$.

At a spatial maximum of $u$, the inequality $dd^c u\leq0$ gives $\omega\leq \hat{\omega}$. At a minimum of $u$, the inequality is reversed. The forcing term in the resulting scalar comparison is
$$
\frac12\log\frac{\hat{\omega}(t)^n}{\eta^n}-\frac12\Lambda_\eta(\hat{\omega}(t)-\eta)
=O(e^{-2\lambda t}).
$$
The order $e^{-2\lambda t}$ follows because the linear term of the logarithmic determinant is $\Lambda_\eta(\hat{\omega}-\eta)$ and cancels the trace term. Applying the scalar maximum principle to $u_t+\lambda u$ at the extrema of $u$, with $u(0)=0$, gives
\begin{equation}\label{appendix_potential_decay}
\|u(t)\|_{C^0}\leq Ce^{-\lambda t}.
\end{equation}
The constant absorbs the bounded forcing on any initial compact time interval.

We next bound $u_t$. In this proof only, write $w=u_t$, and let $P_t$ be the scalar differential operator $P_tv=\Lambda_{\omega(t)}dd^c v/2$. Thus, $P_t$ is the complex Laplacian of the positive Hermitian form $\omega(t)$, with the sign for which $P_tv\leq0$ at a maximum of $v$. For a constant $A>0$, differentiation of~\eqref{appendix_scalar_equation} gives
\begin{equation}\label{appendix_time_derivative}
\begin{aligned}
(\partial_t-P_t)(w-Au)
&=\frac12\Lambda_\omega(\hat{\omega}_t-A\hat{\omega})-(\lambda+A)w+\frac{n}{2}A+f_t,\\
(\partial_t-P_t)(w+Au)
&=\frac12\Lambda_\omega(\hat{\omega}_t+A\hat{\omega})+(A-\lambda)w-\frac{n}{2}A+f_t.
\end{aligned}
\end{equation}
Choose $A>\lambda$ sufficiently large that $\hat{\omega}_t-A\hat{\omega}\leq0$ and $\hat{\omega}_t+A\hat{\omega}\geq c_1\eta$ for a constant $c_1>0$. The first equation in~\eqref{appendix_time_derivative}, together with the bound for $u$, bounds $w$ above by the maximum principle. For the lower bound, the scalar equation and the arithmetic--geometric mean inequality give
$$
\Lambda_\omega\eta
\geq n\left(\frac{\eta^n}{\omega^n}\right)^{1/n}
\geq c_2e^{-2w/n}.
$$
At a sufficiently negative space-time minimum of $w+Au$, the right-hand side of the second equation in~\eqref{appendix_time_derivative} is at least $c_3e^{-2w/n}+(A-\lambda)w-C>0$. This contradicts the parabolic minimum principle. Therefore, $w$ is also bounded below. The scalar equation now gives
\begin{equation}\label{appendix_determinant_bounds}
C^{-1}\leq\frac{\omega(t)^n}{\eta^n}\leq C.
\end{equation}

For the second-order estimate, set $\varphi=2u$ and $S=\Lambda_{\hat{\omega}}\omega$. The equation for $\varphi$ is
$$
\varphi_t=\log\frac{(\hat{\omega}+i\partial\bar\partial\varphi)^n}{\hat{\omega}^n}
-\log\frac{\eta^n}{\hat{\omega}^n}+2f-\lambda\varphi.
$$
Apply the Hermitian trace estimate of~\cite[Section~3]{Gill} with test function $\log S+\exp(B(D-\varphi))$, where $D$ bounds $\varphi$ above and $B>0$ is sufficiently large. The additional term in $S_t$ caused by the time dependence of $\hat{\omega}$ is
$$
\sum_{i,j=1}^n(\partial_t \hat{\omega}^{i\bar j})(\omega_{i\bar j}-\hat{\omega}_{i\bar j}),
$$
where $\hat{\omega}^{1\bar 1},\hat{\omega}^{1\bar 2},\dots,\hat{\omega}^{n\bar n}$ denote the inverse metric coefficients. This term is bounded in absolute value by $C(S+n)$. The determinant bounds give $S\geq c>0$, so its contribution to $\partial_t\log S$ is bounded. The damping contributes the bounded-below term $\lambda(S-n)/S$ to $(P_t-\partial_t)\log S$. The remaining reference coefficients and forcing derivatives are uniformly bounded, and $\varphi_t$ is bounded. The exponential term absorbs the torsion gradient terms in the trace estimate. The maximum principle therefore gives $S\leq C$, which together with~\eqref{appendix_determinant_bounds} proves~\eqref{appendix_uniform_metric}.

The metric bounds control $\Delta_\eta u$, so fixed-background elliptic estimates give a gradient bound. Parabolic H\"older and Schauder estimates then bound all derivatives on positive-time strips, as in~\cite[Sections~4--5]{Gill}. These uniform estimates give continuation for all time. Interpolation with~\eqref{appendix_potential_decay} gives $$\|dd^c u(t)\|_{C^r}\leq Ce^{-\gamma t}$$ for every $0<\gamma<\lambda$, where $C$ may depend on $r$ and $\gamma$. Since $\hat{\omega}(t)-\eta=O(e^{-\lambda t})$ in every smooth norm, formula~\eqref{appendix_scalar_equation} proves~\eqref{appendix_convergence_rate}.
\end{proof}

\subsection{Nonnegative holomorphic bisectional curvature}

We use the K\"ahler curvature convention in which the Fubini--Study metric has positive holomorphic bisectional curvature. For nonzero $(1,0)$-vectors $v,w$, the holomorphic bisectional curvature of $\eta$ in the directions $v,w$ is the curvature component $R_\eta(v,\bar v,w,\bar w)$ divided by $|v|_\eta^2|w|_\eta^2$. Nonnegative holomorphic bisectional curvature means that this quotient is nonnegative for every pair of directions.

\begin{theorem}\label{appendix_nonnegative_bisectional}
Suppose that the fixed compact K\"ahler background $(X,\eta)$ has nonnegative holomorphic bisectional curvature. Then, for every $\lambda>0$ and every smooth positive Hermitian initial form $\omega_0$, the solution of~\eqref{torsion_flow} exists for all time and converges smoothly to $\eta$. The conclusions~\eqref{appendix_uniform_metric} and~\eqref{appendix_convergence_rate} hold.
\end{theorem}

\begin{proof}
Let $\hat{\omega}(t)$ be the reference form~\eqref{appendix_reference}, and choose $0<c<1$ such that $\omega_0>c\eta$. Suppose that $\hat{\omega}-c\eta$ has a null vector for the first time. At a point of first contact, diagonalize $\hat{\omega}$ in an $\eta$-unitary frame, with eigenvalues $a_1=c$ and $a_j\geq c$ for every $j$.

By the K\"ahler Weitzenb\"ock formula, the curvature and damping reaction in the first diagonal component of $\partial_t \hat{\omega}=-\Delta_\eta \hat{\omega}/2-\lambda(\hat{\omega}-\eta)$ is exactly
$$
\sum_{i=1}^nR_{1\bar1i\bar i}(a_i-c)+\lambda(1-c)>0;
$$
see~\cite[Lemma~2.1]{NiNiu} for the heat-operator normalization. The inequality follows from nonnegative holomorphic bisectional curvature and $\lambda>0$. At the first contact the rough diffusion term is nonnegative in the null direction. The strictly positive reaction therefore contradicts the tensor maximum principle. Hence, we have that $\hat{\omega}(t)\geq c\eta$ for all $t\geq0$, and Theorem~\ref{appendix_reference_convergence} proves the assertion.
\end{proof}

In particular, Theorem~\ref{appendix_nonnegative_bisectional} applies to complex projective space with its Fubini--Study form, to flat complex tori equipped with a flat K\"ahler form, and to products of these backgrounds with the product K\"ahler form. The curvature hypothesis concerns the fixed form $\eta$. The theorem does not assert that the evolving Hermitian metric $\omega(t)$ has nonnegative bisectional curvature at every time.

%\section*{Acknowledgements} 

%This research was completed while the author was studying at the Mathematics Institute of the University of Warwick. The author would therefore like to thank the University of Warwick for its hospitality.

%The author is in debt to the Russian mathematical society, especially, to all the professors teaching the "Math in Moscow" program, and most importantly, to the author's supervisor Alexander Petrovich Veselov at Loughborough University, as they cultivated the author's mathematical literacy and maturity.

\section*{Statements and Declarations}

\noindent\textbf{Funding.} No funding was received to assist with the preparation of this manuscript.

\noindent\textbf{Competing interests.} The author declares no relevant financial or non-financial interests.

\noindent\textbf{Data availability.} Data sharing is not applicable to this article.

\noindent\textbf{Use of generative AI.} During the preparation of this manuscript, the author used OpenAI’s ChatGPT to help handle minor details and grammatical issues. All suggested changes were reviewed by the author, who takes full responsibility for the mathematical accuracy and the final content of this article.

\singlespacing

\end{document}